\documentclass[11pt]{amsart}

\usepackage{enumerate,url,amssymb,mathrsfs,upref,pdfsync,nicefrac}

\newtheorem{theorem}{Theorem}[section]
\newtheorem{lemma}[theorem]{Lemma}
\newtheorem*{lemma*}{Lemma}
\newtheorem{proposition}[theorem]{Proposition}
\newtheorem{conjecture}[theorem]{Conjecture}
\newtheorem*{prop*}{Proposition}

\theoremstyle{definition}
\newtheorem{definition}[theorem]{Definition}

\theoremstyle{remark}
\newtheorem{remark}[theorem]{Remark}

\numberwithin{equation}{section}

\newcommand{\abs}[1]{\lvert#1\rvert}
\newcommand{\Abs}[1]{\left\lvert#1\right\rvert}
\newcommand{\norm}[1]{\lVert#1\rVert}

\newcommand{\C}{\mathbb{C}}

\newcommand{\HH}{\mathbb{H}}

\newcommand{\R}{\mathbb{R}}

\newcommand{\Z}{\mathbb{Z}}

\DeclareMathOperator{\dist}{dist}

\DeclareMathOperator{\re}{Re}
\DeclareMathOperator{\im}{Im}

\DeclareMathOperator{\diam}{diam}

\def\XXint#1#2#3{{\setbox0=\hbox{$#1{#2#3}{\int}$}
\vcenter{\hbox{$#2#3$}}\kern-.5\wd0}}

\def\le{\leqslant}
\def\ge{\geqslant}

\begin{document}
\baselineskip6mm

\title{Quasisymmetric graphs and Zygmund~functions}

\author{Leonid V. Kovalev}
\address{Department of Mathematics, Syracuse University, Syracuse, NY 13244, USA}
\email{lvkovale@syr.edu}
\thanks{Kovalev was supported by the NSF grant DMS-0968756. Onninen was supported by the NSF grant DMS-1001620.}

\author{Jani Onninen}
\address{Department of Mathematics, Syracuse University, Syracuse, NY 13244, USA}
\email{jkonnine@syr.edu}




\begin{abstract}
A quasisymmetric graph is a curve whose projection onto a line is a quasisymmetric map.
We show that this class of curves is related to solutions of the reduced Beltrami equation and to
a generalization of the Zygmund class $\Lambda_*$.  This relation makes it possible to use the tools of harmonic
analysis to construct nontrivial examples of quasisymmetric graphs and of  quasiconformal maps.
\end{abstract}

\maketitle

\section{Introduction}

Let $X$ and $Y$ be subsets of a Euclidean space $\R^n$. An embedding $f\colon X\to Y$ is quasisymmetric if there is a homeomorphism
$\eta\colon [0,\infty)\to [0,\infty)$ such that for any triple of distinct points $a,b,x\in X$
\begin{equation}\label{defqs}
\abs{f(x)-f(a)}\le \eta(t)\abs{f(x)-f(b)} \qquad \text{where } \quad t =\frac{\abs{x-a}}{\abs{x-b}}.
\end{equation}

We call a set $\Gamma\subset \C$ a \emph{quasisymmetric graph}  if the orthogonal projection of $\Gamma$ onto $\R$ is a quasisymmetric
homeomorphism between $\Gamma$ (with the metric induced from $\C$) and $\R$.
This should be compared to Lipschitz graphs, which can be defined by requiring the projection to be bi-Lipschitz,
a stronger property than quasisymmetry. For instance, we shall see that the graph of any function in the Zygmund class $\Lambda_*$ is quasisymmetric.

This paper has three main goals.
\begin{enumerate}[(I)]
\item\label{goal1} Parametrize quasisymmetric graphs by homeomorphic solutions of the reduced Beltrami equation;
\item\label{goal2} Use a generalization of the Zygmund class $\Lambda_*$ to construct quasisymmetric graphs;
\item\label{goal3} Use~\eqref{goal1} and~\eqref{goal2} to solve a problem from~\cite{KO} concerning the variation
of reduced quasiconformal maps.
\end{enumerate}

Our success in~\eqref{goal1} is partial in that we can parametrize only quasisymmetric graphs with small distortion. This is made precise with the concept of an $s$-quasisymmetric map introduced by Tukia and V\"ais\"al\"a~\cite{TV}. Namely, the map $f$ in~\eqref{defqs} is called $s$-quasisymmetric (where $s>0$ is a constant)
if $\eta$ can be chosen so that $\eta(t)\le t+s$ for $0\le t\le 1/s$. Observe that any quasisymmetric map
is $s$-quasisymmetric for large enough $s$. The term \emph{$s$-quasisymmetric graph} should be self-explanatory.

\begin{definition}\label{defqc}
A $W_{\rm loc}^{1,2}$-homeomorphism $f\colon \C\to\C$ is quasiconformal if there exists a constant $k\in [0,1)$ such that
\begin{equation}\label{qcmap}
\abs{f_{\bar z}} \le k\abs{f_z} \qquad \text{a.e. in $\C$.}
\end{equation}
\end{definition}
We sometimes refer to the constant $k$ in~\eqref{qcmap} by writing that $f$ is $k$-quasiconformal.
The images of circles and lines under a quasiconformal map are called quasicircles and quasilines, respectively. These curves
are ubiquitous in geometric function theory and still pose challenging problems~\cite{Ge, Me, Pr, Sm}.

Inequality~\eqref{qcmap} is a form of the Beltrami equation $f_{\bar z}=\nu(z)f_z$ where $\norm{\nu}_{L^\infty}<1$.
A closely related equation with $f_z$ replaced by $\re f_z$ (or $\im f_z$) arises from consideration of
elliptic PDE in the plane and generated considerable interest recently~\cite{AN,AIMbook,AJ,GIKMS,IKO2,Ja,KO}.
We state this \emph{reduced Beltrami equation} as an inequality, without an explicit coefficient $\nu$.

\begin{definition}\label{defrqc}
A nonconstant continuous $W_{\rm loc}^{1,2}$-mapping $f\colon \C\to\C$ is \emph{reduced quasiconformal} if there exists a
constant $k\in [0,1)$ such that
\begin{equation}\label{rqcmap}
\abs{f_{\bar z}} \le k\re f_z,\qquad  \text{a.e. in $\C$.}
\end{equation}
\end{definition}

Definition~\ref{defrqc} does not explicitly require $f$ to be a homeomorphism, but the injectivity of $f$ is a consequence of inequality~\eqref{rqcmap}~\cite[Corollary~1.5]{IKO}. In addition,
$f$ maps every horizontal line onto a graph over $\R$~\cite[Proposition 1.5]{IKO2} except for the degenerate case
\begin{equation}\label{triv}
f(z)=i\lambda z+b, \qquad \lambda\in\R , \quad b\in\C,
\end{equation}
when both sides of~\eqref{rqcmap} vanish identically.

We are now ready to state the result that achieves Goal~\eqref{goal1} for graphs of small distortion.

\begin{theorem}\label{smallk}
There exists a constant $s_0>0$ such that any $s$-quasisymmetric graph $\Gamma\subset \C$ with $s<s_0$ is the image of $\R$ under
a reduced quasiconformal mapping $f\colon \C\to \C$. Moreover, the constant $k$ in~\eqref{rqcmap} depends only on $s$ and
$k\to 0$ as $s\to 0$.
\end{theorem}

It should be mentioned that even though $\Gamma$ has a natural quasisymmetric parametrization by $\R$ (the inverse of projection), this parametrization cannot in general be extended to a reduced quasiconformal mapping of $\C$. Instead
we use the parametrization that comes from the conformal map of upper half-plane onto the domain above $\Gamma$.

Our Goal~\eqref{goal2} is achieved by means of Theorem~\ref{suffgraph}. It employs the generalized Zygmund class
$\Lambda_\mu$ which is introduced in Definition~\ref{gzyg}.

\begin{theorem}\label{suffgraph}
Let $\mu$ be a doubling measure on $\R$. Let $u$ and $v$ be real functions on $\R$ such that $u'=\mu$
and $v\in \Lambda_\mu$. Then the image of $\R$ under the map
 $\Gamma(t) = u(t)+iv(t)$ is a quasisymmetric graph.

Furthermore, if the doubling constant of $\mu$ and the $\Lambda_\mu$-seminorm of $v$ are sufficiently small, then $\Gamma(\R)$ is an $s$-quasisymmetric graph where $s$ is small.
\end{theorem}

Theorems~\ref{smallk} and~\ref{suffgraph} from the basis for the proof of our third main result. To state it,
let $\Phi_q\colon [0,\infty)\to [0,\infty)$ be any convex increasing function such that
\begin{equation}\label{gauge}
\Phi_q(t)=\frac{t}{(\log 1/t)^q}\qquad \text{for small $t$.}
\end{equation}
We refer to Definition~\ref{vardef} for the notion of $\Phi$-variation.

\begin{theorem}\label{varrqc}
There exists a reduced quasiconformal mapping $f\colon \C\to\C$ whose restriction to the line segment $[0,1]$
has infinite $\Phi_q$-variation for every $0<q<1$.
\end{theorem}

This result was previously known only for $q<\nicefrac{1}{2}$~\cite[Remark 4.1]{KO}. On the other hand, for $q>1$ every
reduced quasiconformal map has finite $\Phi_q$-variation on line segments~\cite[Theorem 1.7]{KO}. The borderline
case $q=1$ remains open. Using the additivity of reduced quasiconformal maps, one can strengthen the conclusion of Theorem~\ref{varrqc} by replacing one line segment with an arbitrary countable set of lines. See~\cite{KO} for details. The size of such exceptional sets for Sobolev and quasiconformal maps was recently studied in~\cite{BMT}.

We do not know if the restriction $s<s_0$ is necessary in Theorem~\ref{smallk}. The converse statement holds without
such restrictions.

\begin{proposition}\label{conv} If $f\colon \C\to \C$ is a reduced quasiconformal map which is not of the form~\eqref{triv},
then $f(\R)$ is an $s$-quasisymmetric graph with $s=s(k)\to 0$ as $k\to 0$. Here $k$ is the constant in~\eqref{rqcmap}.
In addition,
\begin{equation}\label{addition}
\im f\big|_{\R}\in \Lambda_\mu \qquad\text{where  } \mu=\frac{d}{dx}\re f(x).
\end{equation}
\end{proposition}
 
This leads to a conjecture.

\begin{conjecture}\label{niiice}
The images of $\R$ under reduced quasiconformal maps $\C\to\C$ are precisely quasisymmetric graphs and vertical lines.
\end{conjecture}

Parametrization of Lipschitz graphs is much easier to achieve. They corresponds to delta-monotone maps, which are defined as follows. 
A map $f \colon \C \to \C$ is delta-monotone if there exists a constant $\delta>0$ such that
\[
\re \frac{f(z)-f(\zeta)}{z-\zeta}\ge \delta \, \frac{\abs{f(z)-f(\zeta)}}{\abs{z-\zeta}} \qquad \text{for all distinct }z,\zeta\in \C.
\]
This is a proper subclass of reduced quasiconformal maps~\cite{IKO2}.

\begin{proposition}\label{prodelta}
The images of $\R$ under nonconstant delta-monotone maps $\C\to\C$ are precisely Lipschitz graphs.
\end{proposition}

\begin{remark}
The concept of a quasisymmetric graph also makes sense for $k$-hypersurfaces in $\R^n$, although it reduces to Lipschitz graphs when $2k>n$.
It would be interesting to investigate, e.g., $2$-dimensional quasisymmetric graphs in $\R^4$, but we do not pursue this direction here.
\end{remark}

\subsection*{Acknowledgments} We thank Vladimir Dubinin, Pekka Tukia and Jussi V\"ais\"al\"a  for their helpful comments.

\section{Preliminaries}

By an \emph{embedding} we understand a map that is a homeomorphism onto its image.
An embedding  $\Gamma\colon \R\to\C$ satisfies the Ahlfors condition if there exists a constant $K$ such that
\begin{equation}\label{ahl}
\diam \Gamma([a,b])\le K \abs{\Gamma(a)-\Gamma(b)} \qquad \text{whenever $a<b$.}
\end{equation}
By a classical theorem of Ahlfors~\cite{Ah63}, the condition~\eqref{ahl} characterizes quasilines, i.e., images of lines under quasiconformal maps. Tukia~\cite{Tu2} proved that every quasisymmetric embedding $\R\to\C$ extends to a quasiconformal map $\C\to\C$.
It immediately follows that every quasisymmetric graph is a quasiline. However, a quasiline may be a graph without being a quasisymmetric graph. Such examples are easy to find, e.g., the graphs $y=\sqrt{x^+}$ and $y=e^x$.

The foundational results on $s$-quasisymmetric maps were obtained by Tukia and V\"ais\"al\"a in 1980s. We will use three of them. For simplicity, the theorems are stated here in the planar case.

\begin{theorem}\label{tv54}~\cite[Theorem 5.4]{TV}
There is a number $s_0>0$ such that for $0\le s\le s_0$ any $s$-quasisymmetric embedding of $\R$ into $\C$ extends to
a $s_1$-quasisymmetric mapping $\C\to\C$. Here $s_1=s_1(s)\to 0$ as $s\to 0$.
\end{theorem}

\begin{theorem}\label{tv26}~\cite[Theorem 2.6]{TV}
Any $s$-quasisymmetric homeomorphism $f\colon \C\to\C$ is $k$-quasiconformal with $k=k(s)\to 0$ as $s\to 0$.
Conversely, any $k$-quasiconformal homeomorphism $f\colon \C\to\C$ is $s$-quasisymmetric with $s=s(k)\to 0$ as $k\to 0$.
\end{theorem}

\begin{theorem}\label{va39}~\cite[Theorem 3.9]{Va}
Let $0<\varkappa\le \frac{1}{25}$, and let $f\colon \R\to\C$ be a map such that for any $a<b$ there is
an affine map $h \colon [a,b]\to \C$ with
\begin{equation}\label{va39a}
\sup_{[a,b]}\, \abs{h-f} \le \varkappa \abs{h(a)-h(b)}.
\end{equation}
Then $f$ is $s$-quasisymmetric, where $s=s(\varkappa)\to 0$ as $\varkappa\to 0$.
\end{theorem}

\begin{definition}\label{doubling}
A positive Radon measure $\mu$ on $\R$ is doubling if there exists $\delta>0$ such that
\begin{equation}\label{dbl0}
\mu(I)\le (1+\delta)\mu (J)  
\end{equation}
for any adjacent intervals $I, J$ of equal length.
\end{definition}

\begin{definition} A continuous function $g\colon \R\to\R$ belongs to the Zygmund class $\Lambda_*$
if there exists a constant $M>0$ such that
\begin{equation}\label{zydef}
\abs{g(x+h)-2g(x)+g(x-h)}\le 2Mh \qquad \text{for all $x\in \R$, $h>0$}
\end{equation}
The smallest such $M$ is the Zygmund seminorm of $g$.
\end{definition}

It is often said that~\eqref{zydef} is an additive form of~\eqref{dbl0}.
One can interpret~\eqref{zydef} by saying that the nonlinearity of $g$ on any interval is controlled by the length of the interval. The relevance of the class $\Lambda_*$ to geometric function theory is evident by now~\cite{GS,RC,CCH}.
But our subject required a wider class of functions, in which the length is replaced by a general nonatomic Radon measure on~$\R$. A measure is nonatomic if it gives zero mass to every singleton. All our measures are positive.

\begin{definition}\label{gzyg}
Let $\mu$ be a nonatomic Radon measure on $\R$. A continuous function $g\colon \R\to\R$ belongs to the generalized Zygmund class $\Lambda_\mu$
if there exists a constant $M>0$ such that
\begin{equation}\label{gzydef}
\abs{g(x+h)-2g(x)+g(x-h)}\le M\mu([x-h,x+h]) 
\end{equation}
for all $x\in \R$, $h\ge 0$.
The smallest such $M$ is the seminorm of $g$ in $\Lambda_\mu$.
\end{definition}

We should make precise the  remark about the controlled nonlinearity of $g$. Given distinct points $a,b\in \R$, let
\begin{equation}\label{gab}
g_{ab}(x)=\frac{b-x}{b-a}\,g(a)+\frac{x-a}{b-a}\,g(b)
\end{equation}
denote the affine function that agrees with $g$ at $a$ and $b$. If $g$ satisfies~\eqref{gzydef}, then
\begin{equation}\label{gzyd2}
\sup_{[a,b]}\,\abs{g-g_{ab}}\le M\mu([a,b]) \qquad \text{whenever $a<b$.}
\end{equation}
Indeed, we lose no generality in  assuming that $g(a)=g(b)=0$ and $\abs{g}$ attains its maximum on $[a,b]$ at a point
$\xi\le \frac{a+b}{2}$. Applying~\eqref{gzydef} with $x=\xi$ and $h=\xi-a$, we find
\[
\abs{g(2\xi-a)-2g(\xi)}\le M\mu([a,b]),\quad \text{ hence }\quad \abs{g(\xi)}\le M\mu([a,b]).
\]
Conversely,~\eqref{gzyd2} yields~\eqref{gzydef} with $2M$ in place of $M$.

\begin{definition}\label{vardef}
Let $\Phi\colon [0,\infty)\to[0,\infty)$ be a convex increasing function. A function $v\colon [a,b]\to\R$ has finite $\Phi$-variation if
\begin{equation}\label{finva}
\sup \sum_{j=1}^N \Phi(\abs{v(x_j)-v(x_{j-1})})<\infty,
\end{equation}
where the supremum is taken over all partitions $a=x_0<\dots<x_N=b$ and over all $N\ge 1$.
If $v$ is defined on $\R$, we say that it has locally finite $\Phi$-variation if~\eqref{finva} holds for every bounded interval.
\end{definition}

In the sequel, the constants $C$ and $c$ in estimates may be different from one line to another.

\section{Proof of Propositions~\ref{conv} and~\ref{prodelta}}

Two of the results stated in the introduction admit simple proofs.

\begin{proof}[Proof of Proposition~\ref{conv}]
To a reduced quasiconformal map $f$ we associate the one-parameter family $f_{\lambda}(z)=f(z)+i\lambda z$, $\lambda\in\R$. Unless
$f$ is of the form~\eqref{triv}, each $f_{\lambda}$ is also reduced quasiconformal, as it is nonconstant and satisfies~\eqref{rqcmap} with the same constant as $f$.
Therefore, $f_{\lambda}$ is $\eta$-quasisymmetric with $\eta$ independent of $\lambda$.
In particular, for any triple of distinct points $a,b,x\in \R$ we have
\begin{equation}\label{qs1}
\abs{f_{\lambda}(x)-f_{\lambda}(a)}\le \eta(\abs{\tau})\abs{f_{\lambda}(x)-f_{\lambda}(b)},\qquad \tau=\frac{{x-a}}{x-b}.
\end{equation}
Setting $\lambda = -\im \frac{f(x)-f(b)}{x-b}$ results in
\begin{equation}\label{qs2}
\abs{f(x)-f(a)-i\tau \im (f(x)-f(b))}\le \eta(\abs{\tau})\,\abs{\re (f(x)- f(b))}.
\end{equation}
There are two ways to use~\eqref{qs2}. First, we can take the real part and obtain
\begin{equation}\label{qs3}
\abs{\re (f(x)-f(a))}\le \eta(\abs{\tau})\, \Abs{\re (f(x)- f(b))}
\end{equation}
which simply says that $\re f$ is a quasisymmetric map from $\R$ onto $\R$.
Combining~\eqref{qs3} with the quasisymmetry of $f$, we conclude that the projection $w\mapsto \re w$ is a quasisymmetric map from $\Gamma$ to $\R$.

 Let $\mu$ denote the distributional derivative of $\re f(x)$ with respect to $x$. Since $\re f$ is quasisymmetric, $\mu$ is a doubling
 measure on $\R$~\cite[Remark~13.20b]{He}.
Taking the imaginary part in~\eqref{qs2} yields
\begin{equation}\label{qs4}
\abs{\im (f(x)-f(a))-\tau\im (f(x)-f(b))}\le \eta(\abs{\tau})\, \Abs{\re (f(x)- f(b))}.
\end{equation}
Choosing $x=\frac{a+b}{2}$, we conclude that  $\im f\in\Lambda_\mu$.
\end{proof}

\begin{remark} Every quasisymmetric graph $y=g(x)$ admits a natural quasisymmetric parametrization by $\R$, namely
 $f(x)=x+i g(x)$. In general, this function $f$ does not satisfy~\eqref{addition} and therefore cannot be extended
to a reduced quasiconformal map of the plane. For a concrete example, take  the graph $y=x^{\nicefrac{1}{3}}$.
\end{remark}

\begin{proof}[Proof of Proposition~\ref{prodelta}]

It is obvious that $f(\R)$ is a Lipschitz graph for every delta-monotone map $f \colon \C \to \C$.  Conversely, for any $L$-Lipschitz real function $g$ the mapping
\[f(z):=\re z+iL^2\im z + ig(\re z)\]
satisfies
\begin{equation}\label{lipex1}
\re f_z= \frac{L^2+1}{2},\qquad \abs{\im f_z} \le \frac{L}{2} \le \re f_z,
\end{equation}
and
\begin{equation}\label{lipex2}
\abs{f_{\bar z}} \le \sqrt{\frac{(L^2-1)^2}{4}+\frac{L^2}{4}} \le  k \re f_z
\end{equation}
with $k=k(L)<1$. The combination of~\eqref{lipex1} and~\eqref{lipex2} implies that $f$ is delta-monotone, see~\cite[Lemma 12]{Ko}.
\end{proof}

\section{Proof of Theorem~\ref{smallk}}

Let $\HH=\{z\colon \im z>0\}$ denote the upper half-plane.
By Theorems~\ref{tv54} and~\ref{tv26} the curve $\Gamma$ is a $k$-quasiline where $k$ is small if $s$ is.

The curve $\Gamma$ divides the plane into two domains; let $\Omega$ denote the upper one.
Let $f\colon \HH\to\Omega$ be a conformal mapping such that
$f(\infty)=\infty$ in the sense of boundary correspondence. Since $\Gamma$ is a $k$-quasiline, $f$ extends to $\Gamma$ by continuity. It then extends to the entire plane by quasiconformal reflection, and the extended mapping is
$\frac{2k}{1+k^2}$-quasiconformal~\cite{Ah63}. By Theorem~\ref{tv26} the correspondence $x\mapsto \re f(x)$ is $s_1$-quasisymmetric where $s_1$ is small if $k$ is.

We claim that there exists $\tilde k\in [0,1)$ such that $\tilde k\to 0 $ as $k\to 0$ and
\begin{equation}\label{suf2}
2\im z \, \abs{f''(z)}\le \tilde k\re f'(z)\qquad  \text{for all } z\in\HH.
\end{equation}
Assume~\eqref{suf2} for now and complete the proof of the theorem.

The Koebe $\nicefrac{1}{4}$-theorem~\cite[(I.6.7)]{Le} yields
\begin{equation}\label{bdist}
\im z \, \abs{f'(z)} \le  2 \dist(f(z),\C\setminus f(\HH))  \qquad \text{for all } z\in \HH.
\end{equation}
Hence
\begin{equation}\label{bdist2}
\lim_{z\to\zeta} \im z\, \abs{f'(z)}=0 \qquad \text{for any $\zeta\in \R$.}
\end{equation}

We extend $f$ to $\C$ following the method  that goes back to Ahlfors and Weill~\cite{AW} and was
further developed in~\cite{Ah,AH,Ha}. Namely, we define $F\colon \C\to\C$ by
\begin{equation}\label{ext1}
F(z)= \begin{cases} f(z) \quad &\im z\ge 0; \\
f(\bar z)+(z-\bar z)f'(\bar z) \quad &\im z<0.
\end{cases}
\end{equation}
By virtue of~\eqref{bdist2} the mapping $F$ is continuous in $\C$. For $z\in \C\setminus\overline{\HH}$ we have
\begin{equation}\label{ext2}
F_z = f'(\bar z)\qquad \text{and}\qquad F_{\bar z} = (z-\bar z)f''(\bar z).
\end{equation}
The comparison of~\eqref{suf2} and~\eqref{ext2} shows that $F$ is reduced quasiconformal.
The theorem is proved, modulo~\eqref{suf2}.
\qed

\begin{proof}[Proof of~\eqref{suf2}]
The first step is to observe that $\re f'>0$ in $\HH$. To this end, introduce the function
\[
u_h(z):=\arg (f(z+h)-f(z)) \qquad \text{for a fixed $h>0$.}
\]
Here we choose the branch of $\arg$ so that $\abs{u_h}<\pi/2$ on $\partial \HH$: this is possible because
$f$ extends to a homeomorphism  $f\colon \overline{\HH} \to \overline{\Omega}$ and $\partial \Omega$ is a graph.
The maximum principle implies $\abs{u_h}<\pi/2$ in $\HH$, and letting $h\to 0$ we obtain
the desired conclusion $\re f'>0$.

The harmonic function $u=\re f'$, being positive in $\HH$, admits the Herglotz representation~\cite[Theorem I.3.5]{Ga}
\begin{equation}\label{u1}
u(z)=\beta \im z+ \frac{1}{\pi}\int_{\R}\im \frac{1}{t-z} \, d\mu(t)
\end{equation}
where $\beta \ge 0$ and $\mu$ is a positive measure on $\R$ such that
\begin{equation}\label{u2}
\int_{\R}\frac{1}{1+t^2}\,d\mu(t)<\infty.
\end{equation}

Integration of ~\eqref{u1} yields $\mu([a,b])=\re (f(b)-f(a))$  for any finite interval $[a,b]\subset \R$. Recall that the map $x\mapsto \re f(x)$ is $s_1$-quasisymmetric where $s_1\to 0$ as $k\to 0$. Therefore, the measure $\mu$ satisfies the doubling condition~\eqref{dbl0} where $\delta\to 0$ as $k\to 0$.

To proceed further, we must establish that $\beta=0$ in~\eqref{u1}. To this end, we need the following growth estimate for univalent
functions $F\colon \HH\to\C$:
\begin{equation}\label{grow}
\abs{F(x+iy)}\le \abs{F(i)}+\frac{(y+1)^4}{y^2}\abs{F'(i)} \qquad \text{for }y \ge 1, \quad \abs{x} \le y+1.
\end{equation}
To prove~\eqref{grow}, introduce
\begin{equation}\label{free1}
G(\zeta)=\frac{-i}{2 F'(i)} \left\{F\left(i\frac{1+\zeta}{1-\zeta}\right)-F(i)\right\},\qquad \abs{\zeta}<1,
\end{equation}
and observe that $G(0)=G'(0)-1=0$.
The growth theorem for class $S$~\cite[Theorem~2.6]{Du} asserts that
\begin{equation}\label{free2}
\abs{G(\zeta)} \le \frac{\abs{\zeta}}{(1-\abs{\zeta})^2} =  \frac{\abs{\zeta} (1+\abs{\zeta})^2 }{(1-\abs{\zeta}^2)^2}  \le  \frac{4}{(1-\abs{\zeta}^2)^2}. 
\end{equation}
We set $x+iy= i \frac{1+\zeta}{1-\zeta}$ and observe that $\abs{\zeta}^2 \le \frac{y^2+1}{(y+1)^2}$. Combinng this with~\eqref{free2} and~\eqref{free1} the inequality~\eqref{grow} follows.

We may assume $0\in \partial\Omega$. For $r>0$ let $\Gamma$ be the connected component of the set $\{z\in \C\setminus \Omega \colon \abs{z}=r\}$ that contains the point $-ir$. By virtue of the Ahlfors condition~\eqref{ahl} the length of $\Gamma$ is bounded from below by $cr$, with $c>0$ independent of $r$. Therefore the mapping $z\mapsto z^p$, where $p=\frac{2\pi}{2\pi-c}>1$, is univalent in $\Omega$. This allows us to apply~\eqref{grow} with $F=f^p$ and conclude that
$\abs{f(x+iy)}=O(y^{\nicefrac{2}{p}})$ as $y\to\infty$,  $\abs{x} \le y$. The Cauchy inequality for $f'$ yields
$\abs{f'(iy)}=O(y^{\frac{2}{p} -1})$ as $y\to\infty$. Since the exponent of $y$ is strictly less than $1$, the coefficient $\beta$ in~\eqref{u1} must vanish.

Returning to~\eqref{u1}, we compute
\begin{equation}\label{u3}
f''(z)=2\,\frac{\partial u}{\partial z}(z)=\frac{1}{\pi i}\int_{\R} \frac{1}{(t-z)^2} \, d\mu(t)
\end{equation}
and
\begin{equation}\label{u3b}
\re f'(z)=u(z) = \frac{\im z}{\pi }\int_{\R} \frac{1}{\abs{t-z}^2} \, d\mu(t)
\end{equation}

Thus, the desired inequality~\eqref{suf2} takes the form
\begin{equation}\label{u5}
\Abs{\int_{\R} \frac{1}{(t-z)^2} \, d\mu(t)}\le \frac{\tilde k}{2} \int_{\R}\frac{1}{\abs{t-z}^2}\, d\mu(t)
\end{equation}

The following lemma yields~\eqref{u5}. It is not particularly new; one can find a similar, but less precise, statement in~\cite[p.~157]{DN}.
\end{proof}

\begin{lemma}\label{dbl} For any $\varepsilon>0$ there exists $\delta>0$ such that the following holds.
If $\mu$ satisfies the doubling condition~\eqref{dbl0} then
\begin{equation}\label{dbl1}
\Abs{\int_{\R} \frac{1}{(t-z)^2} \, d\mu(t)}\le \varepsilon \int_{\R}\frac{1}{\abs{t-z}^2}\, d\mu(t)<\infty
\qquad \text{for all $z\in \HH$.}
\end{equation}
\end{lemma}

\begin{proof} We write $\abs{I}$ for the length of an interval $I$.
Repeated application of the doubling property yields the  growth/decay estimate
\begin{equation}\label{dbl2}
(1-\gamma) \min(\tau, \tau^{-1})^\gamma \le
\frac{\mu(I)}{\tau \mu(J)} \le (1+\gamma) \max(\tau, \tau^{-1})^\gamma,\qquad \tau=\frac{\abs{I}}{\abs{J}}
\end{equation}
for any two intervals  $I$ and $J$ with  a common point.
Here $\gamma\in (0,1)$ depends only on $\delta$, and $\gamma\to 0$ as $\delta\to 0$.

Using shift, scaling, and normalization, we reduce~\eqref{dbl1} to the case $z=i$ and $\mu([-1,1])=1$. By virtue of~\eqref{dbl2},
 for all $t>0$ we have
\begin{equation}\label{dbl3}
(1-\gamma)t \min(t, t^{-1})^\gamma \le \mu([-t,t]) \le (1+\gamma)t \max(t, t^{-1})^\gamma .
\end{equation}
For small $\gamma$ the estimates~\eqref{dbl2} yield the following uniform bounds in $t$,
\begin{equation}\label{dbl3a}
\abs{\mu([-t,t])-t} = \begin{cases} O(\gamma)\quad &\text{if $0<t<1$} \\
 O(\gamma t^{1+\gamma}(1+\log t))\quad &\text{if $t>1$}
 \end{cases}
\end{equation}

We proceed to estimate both sides of~\eqref{dbl1} via integration by parts followed by~\eqref{dbl3a}.
\begin{equation}\label{dbl4}
\begin{split}
\int_{\R}\frac{1}{t^2+1}\, d\mu(t) &= \int_0^\infty \frac{2t}{(t^2+1)^2} \mu([-t,t])\,dt \\
&=\pi+ \int_0^\infty \frac{2t}{(t^2+1)^2} (\mu([-t,t])-t)\,dt
\end{split}
\end{equation}
which in view of~\eqref{dbl3a} implies
\begin{equation}\label{dbl5}
\int_{\R}\frac{1}{t^2+1}\, d\mu(t) = \pi +O(\gamma) \qquad \text{as $\gamma\to 0$.}
\end{equation}
Next,
\begin{equation}\label{dbl6}
\begin{split}
\re \int_{\R} \frac{1}{(t-i)^2} \, d\mu(t) &=  \int_{\R} \frac{t^2-1}{(t^2+1)^2} \, d\mu(t) \\
&= \int_0^\infty \frac{2t(t^2-3)}{(t^2+1)^3} \mu([-t,t])\,dt \\
& = \int_0^\infty \frac{2t(t^2-3)}{(t^2+1)^3} (\mu([-t,t])-t)\,dt \\
& = O(\gamma)
\end{split}
\end{equation}
Finally,
\begin{equation}\label{dbl7}
\begin{split}
\im \int_{\R} \frac{1}{(t-i)^2} \, d\mu(t) &=  \int_{\R} \frac{2t}{(t^2+1)^2} \, d\mu(t) \\
&= \int_{0}^\infty \frac{2(3t^2-1)}{(t^2+1)^3}(\mu([0,t])-\mu([-t,0])) \, dt \\
& \le \delta \int_{0}^\infty \frac{2(3t^2-1)}{(t^2+1)^3}\mu([0,t])\,dt \\
& = O(\delta)
\end{split}
\end{equation}
The combination of~\eqref{dbl5}--\eqref{dbl7} proves~\eqref{dbl1}.
\end{proof}

\section{Proof of Theorem~\ref{suffgraph}}

By virtie of the doubling condition, the map $u\colon \R\to\R$ is $s$-quasisymmetric where $s$ is small if $\delta$ is small.
Thus we may consider the map $t\mapsto \Gamma (t) $ instead of the projection $\Gamma(t)\mapsto u(t)$.

Fix $a,b\in\R$, $a<b$.  The growth estimate for $\mu$,~\eqref{dbl2}, yields
\begin{equation}\label{fam1}
\sup_{[a,b]}\, \abs{u-u_{a,b}} \le C (u(b)-u(a))
\end{equation}
where $C=C(\delta)\to 0$ as $\delta\to 0$. On the other hand, the definition of $\Lambda_\mu$ implies
\begin{equation}\label{fam2}
\sup_{[a,b]}\, \abs{v-v_{a,b}} \le \norm{v}_{\Lambda_\mu} (u(b)-u(a)).
\end{equation}

When $\delta$ and $\norm{v}_{\Lambda_\mu}$ are small, Theorem~\ref{va39} implies that
$\Gamma$ is $s$-quasisymmetric with small~$s$.

Without the smallness condition, we can still conclude from~\eqref{fam1}--\eqref{fam2} that
\begin{equation}\label{suff0}
\abs{\Gamma(x)- \Gamma_{a,b}(x)} \le  K (u(b)-u(a)) , \qquad x\in [a,b],
\end{equation}
with $K$ independent of $a,b$. We shall  demonstrate the existence of a constant $H$ such that
\begin{equation}\label{suff1}
\abs{\Gamma (x)-\Gamma (a)}\le H\abs{\Gamma (x)-\Gamma (b)} \qquad \text{whenever } \abs{x-a}\le \abs{x-b}.
\end{equation}
The property~\eqref{suff1} implies the quasisymmetry of $\Gamma$~\cite[Theorem 10.19]{He}.
We split the proof of~\eqref{suff1} in two cases. If $a\le x\le b$, then~\eqref{suff0} yields
\[
\abs{\Gamma(x)- \Gamma(a)}  \le
 \frac{x-a}{b-x} \abs{\Gamma(x)- \Gamma(b)} + K(u(b)-u(a)).
\]
Since $\abs{x-a} \le \abs{x-b}$, the doubling condition implies 
\[u(b)-u(a) \le (2+\delta) \big(u(b)-u(x)\big),\]
hence 
\[\abs{\Gamma(x)- \Gamma(a)} \le [1+K(2 + \delta)] \abs{\Gamma(x)- \Gamma(b)}. \]
The other case to consider
is $x<a<b$. Now
\[
\abs{\Gamma(a)- \Gamma_{x,b}(a)} \le  K (u(x)-u(b)) \le K \abs{\Gamma (x)-\Gamma (b)}
\]
and
\[ \abs{\Gamma(x)- \Gamma_{x,b}(a)} \le \abs{\Gamma (x)-\Gamma (b)}. \]
Hence
\[
\abs{\Gamma(x)- \Gamma(a)}  \le  (K+1)  \abs{\Gamma(x)-\Gamma(b)}
\]
from which~\eqref{suff1} follows.
\qed

\section{Generalized variation of Zygmund functions}

Any function in the Zygmund class $\Lambda_*$ has a modulus of continuity of the form $C\delta \log (1/\delta)$ on every
finite interval~\cite[Theorem~II.3.4]{Zybook}. The example $g(x)=x\log x$ demonstrates that this modulus of continuity is best possible.
However, at most points the local modulus of continuity can be improved to $C\delta \sqrt{\log (1/\delta) \log \log (1/\delta)}$, see~\cite[Theorem~1]{AP}. Such an improvement is also possible on the average, i.e.,
in terms of generalized variation. This fact may be known, but being unable to find a reference, we give a proof.

\begin{proposition}\label{known}
Any function of class $\Lambda_*$ has locally finite $\Phi_q$ variation for every $q>1/2$.
Here  $\Phi_q$  is the gauge function from~\eqref{gauge}.
\end{proposition}

We need a lemma.

\begin{lemma}\label{elementary}~\cite[Lemma 3.4]{KO}. If a function $g\colon [a,b]\to \R$ satisfies
\[
\sum_{j=1}^N \abs{g(x_j)-g(x_{j-1})} \le C \log^{p}(N+1)
\]
for any partition $a=x_0<\dots<x_N=b$, then $g$ has finite $\Phi_q$ variation for every $q>p$.
\end{lemma}

\begin{proof}[Proof of Proposition~\ref{known}]  Let $g\in \Lambda_*$. We claim that there exists a constant $C$ such that for any triple $a<x<b$
\begin{equation}\label{gr4new}
\frac{(g(x)-g(a))^2}{x-a} + \frac{(g(x)-g(b))^2}{b-x} \le \frac{(g(b)-g(a))^2}{b-a} + C (b-a).
\end{equation}
Using the linear interpolant~\eqref{gab} we rewrite the left-hand side of~\eqref{gr4new} in terms of the difference $\delta:=g(x)-g_{ab}(x)$:
\[
\begin{split}
 \frac{(g(x)-g(a))^2}{x-a} &+ \frac{(g(x)-g(b))^2}{b-x}\\
&= \frac{\delta^2}{x-a}+\frac{\delta^2}{b-x} + \frac{(g_{ab}(x)-g(a))^2}{x-a} + \frac{(g_{ab}(x)-g(b))^2}{b-x} \\
& = \frac{\delta^2}{x-a}+\frac{\delta^2}{b-x} + \frac{(g(b)-g(a))^2}{b-a}
\end{split}
\]
It remains to prove that
\begin{equation}\label{gr8}
\frac{\delta^2}{\min(x-a,b-x)} \le C (b-a).
\end{equation}
Recall that $\delta\le C(b-a)$ by~\eqref{gzyd2}.
This immediately implies~\eqref{gr8} when $(x-a)$ is comparable to $(b-x)$. If $x$ is very close
to, say, $a$, then we use the  log-Lipschitz estimate $\delta\le C(x-a)|\log (x-a)|$, see~\cite[Proposition~1]{CD}.
Thus~\eqref{gr8} holds in either case.

Repeated application of~\eqref{gr4new} shows that for any partition $x_0,\dots, x_N$ of the interval $[a,b]$ we have
\begin{equation*}
\sum_{j=1}^N \frac{\abs{g(x_j)-g(x_{j-1})}^2}{x_j-x_{j-1}} \le C \log (N+1).
\end{equation*}
where $C$ is independent of $N$. The Cauchy-Schwarz inequality yields
\begin{equation*}
\sum_{j=1}^N \abs{g(x_j)-g(x_{j-1})} \le C \log^{1/2}(N+1),
\end{equation*}
and Lemma~\ref{elementary} completes the proof.
\end{proof}

Turning to the generalized Zygmund class $\Lambda_\mu$, we immediately find that the modulus of continuity is not log-Lipschitz in general. Indeed, $\Lambda_\mu$ always contains an antiderivative of $\mu$.
On the other hand, a version of Proposition~\ref{known} holds in this generality, albeit with a worse exponent.

\begin{proposition}\label{notknown} Let $\mu$ be a nonatomic Radon measure on $\R$.
Any function of class $\Lambda_\mu$ has locally finite $\Phi_q$ variation for every $q>1$.
\end{proposition}

\begin{proof}  Let $g\in \Lambda_\mu$. We claim that there exists a constant $C$ such that for any triple $a<x<b$
\begin{equation}\label{abc}
\abs{g(x)-g(a)}+\abs{g(x)-g(b)} \le \abs{g(a)-g(b)} + C \mu([a,b]).
\end{equation}
Indeed, in terms of the linear interpolant~\eqref{gab} we have
\[
\begin{split}
\abs{g(x)-g(a)}& +\abs{g(x)-g(b)}\\ &\le \abs{g_{ab}(x)-g(a)}+\abs{g_{ab}(x)-g(b)} + 2\abs{g(x)-g_{ab}(x)} \\
&= \abs{g(a)-g(b)} + 2\abs{g(x)-g_{ab}(x)}
\end{split}
\]
where the last term is controlled by $\mu([a,b])$ by the definition of $\Lambda_\mu$.

Consider a partition $a=x_0<\dots<x_N=b$ where $N=2^m$. Applying~\eqref{abc} to the triples like $x_0,x_1,x_2$, we
obtain
\[
\sum_{j=1}^{2^m} \abs{g(x_j)-g(x_{j-1})} \le C\mu([a,b])+ \sum_{j=1}^{2^{m-1}} \abs{g(x_j)-g(x_{j-1})}
\]
After $m$ iterations of this process the estimate becomes
\[
\sum_{j=1}^{2^m} \abs{g(x_j)-g(x_{j-1})} \le  Cm \mu([a,b]) + \abs{g(a)-g(b)}.
\]
Thus, for any $N$ point partition of $[a,b]$ we have the estimate
\begin{equation}\label{not2}
\sum_{j=1}^{N} \abs{g(x_j)-g(x_{j-1})} \le C\log(N+1)
\end{equation}
where $C$ is independent of $N$. An application of Lemma~\ref{elementary} completes the proof.
\end{proof}

In the next section we prove that Proposition~\ref{notknown} is essentially sharp, even if the measure $\mu$
is assumed to be doubling with a small constant.

\section{Infinite generalized variation}

The principal result of this section concerns the class $\Lambda_\mu$ for singular measures $\mu$.

\begin{theorem}\label{hard}
Let $\delta>0$. There exists a Radon measure $\mu$ on $\R$ with the doubling property~\eqref{dbl0} such that
the class $\Lambda_\mu$ contains a function which has infinite $\Phi_q$-variation on $[a,b]$ for any $0<q<1$ and any $a<b$.
\end{theorem}

Together with previous results this quickly yields Theorem~\ref{varrqc}.

\begin{proof}[Proof of Theorem~\ref{varrqc}]
We use the function $v\in \Lambda_\mu$ provided by Theorem~\ref{hard}, scaling it down to make the $\Lambda_\mu$ seminorm of $v$ as small as needed for Theorem~\ref{suffgraph}.  Then use Theorem~\ref{smallk} to produce
the desired reduced quasiconformal map.
\end{proof}

\begin{proof}[Proof of Theorem~\ref{hard}] Consider $4$-adic intervals
\[
I_{n,j} =\{x\colon 0\le 4^nx-j<1\} = \bigg[\frac{j}{4^n}, \frac{j+1}{4^n}\bigg),\qquad n=1,2,\dots, \ j\in \Z,
\]
and define, for $n\ge 1$, the Rademacher-type functions
\[
\rho_n(x) = \begin{cases}  0, &\qquad x\in I_{n,j}, \quad j\equiv 0,3 \mod 4 \\
1, &\qquad x\in I_{n,j}, \quad j\equiv 1 \mod 4 \\
-1, &\qquad x\in I_{n,j}, \quad j\equiv 2 \mod 4
\end{cases}
\]
For future references we record several properties of the family $\{\rho_n\}$.
\begin{enumerate}[(i)]
\item $\rho_n$ is constant on $I_{m,j}$ when $m\ge n$;
\item $\rho_n$ has zero mean on $I_{m,j}$ when $m<n$.
\item\label{rho3} the set of discontinuities of $\rho_n$ is $\{ j\,  4^{-n} \colon n\ge 1, \ 4\nmid j \}$;
\item\label{rho4} if $\rho_n$ is discontinuous at $x$, then $\rho_m(y)=0$ whenever $m>n$ and $\abs{x-y}<4^{-m}$;
\item\label{rho5} the antiderivative $R_n(x):=\int_0^x \rho_n(t)\,dt$ is $4^{1-n}$-periodic and $\abs{R_n}\le 4^{-n}$;
\item\label{rho6} the product $R_n \rho_m$ is continuous on $\R$ provided that $m<n$;
\item\label{rho7} if $\Psi$ is a function of $\rho_1,\dots,\rho_{n-1},\rho_{n+1},\dots \rho_{m}$, then
\[
\int_{0}^{1} \Psi(x)\,dx = 4\int_{[0,1]\cap\{\rho_n=1\}} \Psi(x)\,dx.
\]
\item\label{rho8} Under the assumptions of~\eqref{rho7}, $\int_{0}^{1} \rho_n(x)\Psi(x)\,dx=0$.
\end{enumerate}

Fix a number $\gamma \in (0,1)$  and define for $n\ge 1$
\[
v_n(x)=\prod_{k=1}^n (1+\gamma \rho_{2k-1}(x))
\]
The measures $v_n(x)\, dx$ have a weak${}^*$ limit, denoted $\mu$.
It is routine to check that $\mu$ satisfies the doubling condition~\eqref{dbl0} where $\delta\to 0$
as $\gamma\to 0$.  Indeed, the weights $v_n$ are doubling with a uniformly controlled constant, and
$\mu(I)$  can be compared to $\int_I v_n$ as long as the length of $I$ is comparable to  $4^{-2n}$. See~\cite{Tu}.

Let us introduce
\begin{equation}\label{defineg}
g(x)=\sum_{n=1}^\infty R_{2n}(x) v_n(x)
\end{equation}
where $R_{2n}$ is the antiderivative of $\rho_{2n}$. Each summand is continuous by virtue of~\eqref{rho6}.
The property~\eqref{rho5} ensures that the series converges  uniformly and at an exponential rate.

\textbf{Step 1: $g\in\Lambda_\mu$.}  For this we will show that~\eqref{gzyd2} holds
for all  $a,b\in \R$  such that $a<b$. Since $g$ is bounded, it suffices  to consider the case $b-a<\nicefrac{1}{16}$.
Let $m$ be the greatest integer such that
\begin{equation}\label{sizeint}
b-a < 4^{-2m}.
\end{equation}
By virtue of~\eqref{rho5} the difference between $g$ and the partial sum
\[
g_m(x) = \sum_{n=1}^m R_{2n}(x) v_n(x)
\]
on the interval $[a,b]$ does not exceed
\[
\bigg(\sup_{[a,b]} v_m\bigg)  \sum_{n>m} 4^{-2n} (1+\gamma)^{n-m} \le C\, 4^{-2m} \sup_{[a,b]} v_m
\le C \mu([a,b]).
\]
Therefore, it suffices to prove the desired property~\eqref{gzyd2} for $g_m$. Differentiation of $g_m$ yields
\begin{equation}\label{gmder}
g_m'(x)=\sum_{n=1}^m \rho_{2n}(x) v_n(x)
\end{equation}
because $v_n$ is locally constant on the support of $R_{2n}$. If $g_m'$ is constant on $[a,b]$ then we are done.
Suppose otherwise. By virtue of~\eqref{rho3} the set of discontinuities of  $g_m'$ is a subset of
$\{j\, 4^{-2n} \colon 1\le n\le m, \ 4\nmid j\}$. Therefore $g_m'$ has exactly one point of discontinuity
on $[a,b]$, say $\theta =\ell \cdot 4^{-2r}$, $ 4\nmid \ell$. The oscillation of  $g_m'$ at this point is at most
$2 v_r(\theta)$. The property~\eqref{rho4} implies that $v_{m}(x)\equiv v_r(\theta)$ for $x\in [a,b]$.
Hence, the deviation of $g_m$ from an affine function on the interval $[a,b]$ does not exceed
\[
2 v_r(\theta) (b-a) =2 \int_a^b v_{m}(x)  \le C \mu([a,b])
\]
as desired.

\textbf{Step 2: the variation of $g$.} Fix $0<q<1$. We must show that $g$ has infinite $\Phi_q$-variation on every $4$-adic interval. It suffices to consider the interval $[0,1]$.
Note that $g$ coincides with the partial sum $g_m$ at all points of the form $j\,  4^{-2m}$, $j\in \Z$. Hence
\begin{equation}\label{gmder1}
\sum_{j=1}^{4^{2m}} \abs{g(j\,  4^{-2m}) - g((j-1)\,  4^{-2m}) } \ge \int_0^1 \abs{ g_m'(x)  }\,dx.
\end{equation}

Let $v_m^*=\max(v_1,\dots,v_m)$.
For $\lambda>0$ and $k=1,\dots,m$ define
\[
E_k(\lambda) =\{x\in [0,1] \colon v_k(x)=v_m^*(x)=\lambda, \ v_n(x)<\lambda \text{ for } n<k\}.
\]
By definition, the sets $E_k(\lambda)$ from a finite partition of the interval $[0,1]$.
We claim that
\begin{equation}\label{gmder2}
\int_{E_k(\lambda)} \abs{ g_m'(x)  }\,dx \ge \frac{\lambda}{4} \abs{E_k(\lambda)},
\end{equation}
where $\abs{\cdot}$ denotes the Lebesgue measure. To this end, restrict the set of integration to $E'=E_{k}(\lambda) \cap \{\rho_{2k}=1\}$. The property~\eqref{rho7}
implies $\abs{E'} = \nicefrac{1}{4}  \abs{E_k(\lambda)}$.  According to~\eqref{rho8}, 
\[
\int_{E'} \rho_{2n} v_n = \begin{cases} \lambda \abs{E'} \quad & \mbox{ if } n=k \\
0 & \text{ otherwise.}  \end{cases}
\]
From~\eqref{gmder} we obtain
\[
\int_{E'} \abs{ g_m'(x) }\,dx \ge \int_{E'} { g_m'(x) }\,dx =\lambda \abs{E'} = \frac{\lambda}{4} \abs{E_k(\lambda)}
\]
which proves~\eqref{gmder2}.

Summing~\eqref{gmder2} over all $k=1,\dots,m$ and all $\lambda>0$ yields
\begin{equation}\label{gmder3}
\int_0^1  \abs{ g_m'(x)  }\,dx \ge \frac{1}{4} \int_0^1 v_m^*(x)\,dx.
\end{equation}
We need a lemma, the proof of which is postponed to the end of this section.

\begin{lemma}\label{maxlem} There exists a positive constant $c>0$ such that
\begin{equation}\label{max1}
\int_0^1 v_m^*(x)\,dx \ge cm ,\qquad m=1,2,\dots
\end{equation}
\end{lemma}

From ~\eqref{gmder1},~\eqref{gmder3} and~\eqref{max1} it follows that
\[
\sum_{j=1}^{4^{2m}} \Abs{g(j\, 4^{-2m}) - g((j-1)\, 4^{-2m}) } \ge  {cm}, \qquad m=1,2,\dots
\]
Jensen's inequality yields
\[
\sum_{j=1}^{4^{2m}} \Phi_q \left(\Abs{g(j\, 4^{-2m}) - g((j-1)\, 4^{-2m}) }\right) \ge
4^{2m} \Phi_q\Big(\frac{c m}{4^{2m}}\Big)   \sim m^{1-q}  \to \infty
\]
as $m\to \infty$.
\end{proof}

\begin{proof}[Proof of Lemma~\ref{maxlem}]
Introduce the random variables
\[X_k= \log (1+\gamma \rho_k) - \frac{1}{4}\log (1-\gamma^2)\]
with $[0,1]$ being the probability space. Since $X_k$ are independent, identically distributed, and have zero mean,
the large deviation bound (Bernstein's inequality~\cite[Theorem 5.11.4]{GSb}) yields
\begin{equation}\label{bern}
\mathbf P \left\{\sum_{k=1}^m X_{2k-1}> \log\frac{1}{4} - \frac{m}{4}\log (1-\gamma^2) \right\} \le e^{-cm}
\end{equation}
where $c>0$ depends only on $\gamma$. An equivalent form of~\eqref{bern} is
\begin{equation}\label{bern2}
\Abs{ \{x\in [0,1]\colon v_m\ge \nicefrac{1}{4}  \} }\le e^{-cm}.
\end{equation}
For $\lambda\ge 1$  let $A(\lambda)=\{x\in [0,1]\colon v_m(x)\ge \lambda\}$. The estimate~\eqref{bern2} yields
\begin{equation}\label{bern3}
\int_{[0,1]\setminus A(\lambda)} v_m \le \frac{1}{4} + \lambda e^{-cm}.
\end{equation}
The right-hand side of~\eqref{bern3} is less than $\nicefrac{1}{2}$ provided that
$\lambda \le \frac{1}{4}\, e^{cm}$. Hence
\begin{equation}\label{asy}
\int_{A(\lambda)} v_m \ge\frac{1}{2},\qquad 1\le \lambda \le \frac{1}{4} e^{cm}.
\end{equation}
Recall a lower  bound for maximal function~\cite[p.~32]{Stb}
\begin{equation}\label{bern4}
\Abs{\{x\in [0,1]\colon v_m^*(x) \ge c_1\lambda\}}\ge \frac{c_2}{\lambda}  \int_{A(\lambda)} v_m
\end{equation}
with universal constants $c_1,c_2>0$.
Integrating~\eqref{bern4} with respect to $\lambda$ and using~\eqref{asy}, we arrive at~\eqref{max1}.
\end{proof}

\bibliographystyle{amsplain}

\end{document}